\documentclass[11pt,reqno]{amsart}
\usepackage{url}
\usepackage{amssymb}
\usepackage[active]{srcltx} % SRC Specials

\newtheorem{thm}{Theorem}[section]
\newtheorem{cor}{Corollary}[section]

\newtheorem{lem}{Lemma}[section]

\theoremstyle{definition}

\theoremstyle{remark}
\newtheorem{rem}{Remark}[section]
\numberwithin{equation}{section}
\begin{document}

\title[{Arithmetic Properties of Overpartition Triples}]
 {Arithmetic Properties of Overpartition Triples}

\author[Liuquan Wang]
{Liuquan Wang}

\address{Department of Mathematics, National University of Singapore, Singapore, 119076}

\email{wangliuquan@u.nus.edu; mathlqwang@163.com}



%\date{May 8, 2014.}

\dedicatory{}

% -----------------------------------------------------------

\begin{abstract}
 Let ${{\overline{p}}_{3}}(n)$ be the number of overpartition triples of $n$. By elementary series manipulations, we establish some congruences for ${\overline{p}}_{3}(n)$ modulo small powers of 2, such as
  \[{{\overline{p}}_{3}}(16n+14)\equiv 0 \pmod{32}, \quad {{\overline{p}}_{3}}(8n+7)\equiv 0 \pmod{64}.\]
We also find many arithmetic properties for ${{\overline{p}}_{3}}(n)$  modulo  7,  9 and 11, involving the following infinite families of Ramanujan-type congruences: for any integers $\alpha \ge 1$ and $n \ge 0$, we have ${{\overline{p}}_{3}}\big({{3}^{2\alpha +1}}(3n+2)\big)\equiv 0$ (mod $9\cdot 2^4$), $\overline{p}_{3}(4^{\alpha-1}(56n+49)) \equiv 0$ (mod 7) and
\[{{\overline{p}}_{3}}\big({{7}^{2\alpha +1}}(7n+3)\big)\equiv {{\overline{p}}_{3}}\big({{7}^{2\alpha +1}}(7n+5)\big)\equiv {{\overline{p}}_{3}}\big({{7}^{2\alpha +1}}(7n+6)\big)\equiv 0 \pmod{7}.\]
\end{abstract}

% -----------------------------------------------------------
\maketitle
% -----------------------------------------------------------

\section {Introduction}
An overpartition $\omega $ of $n$ is a partition of $n$ in which the first occurrence of a part may be overlined.  An overpartition $k$-tuple $({{\omega }_{1}},{{\omega }_{2}},\cdots ,{{\omega }_{k}})$ of $n$ is a $k$-tuple of overpartitions where the sum of all of the parts is $n$. We denote by ${{\overline{p}}_{k}}(n)$ the number of overpartition $k$-tuples of a positive integer $n$ and we define ${{\overline{p}}_{k}}(0)=1$ for convenience. In particular, for $k=1$,
${{\overline{p}}_{1}}(n)$ is denoted as $\overline{p}(n)$. For $k=2$, ${{\overline{p}}_{2}}(n)$ is often denoted  as $\overline{pp}(n)$.

We recall that two of Ramanujan's theta functions are
\begin{displaymath}
\begin{split}
\varphi (q)&=\sum\limits_{n=-\infty }^{\infty }{{{q}^{{{n}^{2}}}}}=\frac{{{(-q;-q)}_{\infty }}}{{{(q;-q)}_{\infty }}}=\frac{(-q;-q)_{\infty }^{2}}{{{({{q}^{2}};{{q}^{2}})}_{\infty }}}, \\
\psi (q)&=\sum\limits_{n=0}^{\infty }{{{q}^{n(n+1)/2}}}=\frac{({{q}^{2}};{{q}^{2}})_{\infty }^{2}}{{{(q;q)}_{\infty }}},
\end{split}
\end{displaymath}
 where ${{(a;q)}_{\infty }}=\prod\nolimits_{n\ge 0}{(1-a{{q}^{n}})}$ is  standard $q$ series notation. It is not difficult to find that the generating function of ${{\overline{p}}_{k}}(n)$ is
\[\sum\limits_{n\ge 0}{{{\overline{p}}_{k}}(n){{q}^{n}}}={{\Big(\frac{{{(-q;q)}_{\infty }}}{{{(q;q)}_{\infty }}}\Big)}^{k}}=\frac{1}{\varphi {{(-q)}^{k}}}.\]

Arithmetic properties of $\overline{p}(n)$ and $\overline{pp}(n)$ have been investigated by many mathematicians. Numerous results have been obtained,  see \cite{Bring}--\cite{Lovejoy} and \cite{Sellers05}-\cite{Wang}. Here we shall mention some  research which motivates this work.

By using 2-adic expansions of $\overline{p}(n)$ or the theory of modular forms, people found many congruences for $\overline{p}(n)$ modulo powers of 2. For example,  Mahlburg \cite{Mahlburg} showed that $\overline{p}(n)\equiv 0$ (mod 64) holds for a set of integers of arithmetic density 1. Kim \cite{Kim1} gave the modulo 8 determination of $\overline{p}(n)$. He \cite{Kim3} also proved that $\overline{p}(n)$ is divisible by 128 for almost all positive integers. Similar results for $\overline{pp}(n)$ have also been found, see \cite{Kim2} for details. There are also some interesting Ramanujan-type congruences modulo small powers of 2. For example, Chen and Lin \cite{Chen} proved that
\[\overline{pp}(8n+7)\equiv 0  \pmod{64}.\]
Keister, Sellers and Vary \cite{Keister} showed that if $k={{2}^{m}}r$ where $m\ge 0$, and $r\ge 1$ is an odd integer, then
\begin{equation}\label{mod2}
{{\overline{p}}_{k}}(n)\equiv \left\{ \begin{array}{ll}
   {{2}^{m+1}} &\textrm{if $n$ is a square or twice a square} \\
 0  & \textrm{otherwise}\\
\end{array} \right.   \pmod{{{2}^{m+2}}}.
\end{equation}

Following their steps, in Section 2, we give the modulo 16 determination of $\overline{p}_{3}(n)$. Meanwhile, we establish some congruences for ${\overline{p}}_{3}(n)$ modulo small powers of 2. For example, we have
\[{{\overline{p}}_{3}}(16n+14)\equiv 0 \pmod{32}, \quad {{\overline{p}}_{3}}(8n+7)\equiv 0 \pmod{64}.\]

In 2005, Hirschhorn and Sellers \cite{Sellers05} obtained an  infinite family of congruences for $\overline{p}(n)$  modulo 3: for any nonnegative integers  $n$ and $\alpha$,
\[\overline{p}({{9}^{\alpha }}(27n+18))\equiv 0 \pmod{3}.\]

Chen and Lin \cite{Chen} gave two infinite families of congruences for $\overline{pp}(n)$ modulo 3:  for any  integers $\alpha \ge 1$ and $n\ge 0$,
\[\overline{pp}({{9}^{\alpha }}(3n+1))\equiv \overline{pp}({{9}^{\alpha }}(3n+2))\equiv 0 \pmod{3}.\]

 Inspired by their work, in Section 3, we will present an infinite family of congruences for ${{\overline{p}}_{3}}(n)$  modulo 9. For any integers $\alpha \ge 1$ and $n\ge 0$, we will show that
\[{{\overline{p}}_{3}}({{3}^{2\alpha +1}}(3n+2))\equiv 0 \pmod{9\cdot 2^4}.\]

It seems much harder to obtain congruences for ${{\overline{p}}_{k}}(n)$ modulo primes other than 2 and 3. The methods used in this direction are quite different. We mention two results here. By using $(p,k)$-parametrization of theta functions,  Chen and Xia \cite{Xia} proved a conjecture of Hirschhorn and Sellers \cite{Hirschhorn05}:
\[\overline{p}(40n+35)\equiv 0  \pmod{40}.\]
Their proof is relatively long and two different short proofs of this conjecture were then given by Lin \cite{Lin} and the author \cite{Wang}.

In 2012, Chen and Lin \cite{Chen} found two infinite families of congruences satisfied by ${\overline{pp}}(n)$. For any integers $\alpha \ge 1$ and $n\ge 0$, we have
\[\overline{pp}({{5}^{\alpha }}(5n+2))\equiv \overline{pp}({{5}^{\alpha }}(5n+3))\equiv 0 \pmod{5}.\]

For nonnegative integers $n$ and $k$,  let ${{r}_{k}}(n)$  denote the number of representations of $n$ as sum of $k$ squares.  In Section 4, we establish the following two arithmetic relations: For any integer $n\ge 1$,
\[{{\overline{p}}_{3}}(7n)\equiv {{(-1)}^{n}}{{r}_{3}}(n) \pmod{7},\]
and
\[{{\overline{p}}_{3}}(11n)\equiv {{(-1)}^{n}}{{r}_{7}}(n) \pmod{11}.\]
From these two relations we deduce many interesting congruences such as $\overline{p}_{3}(4^{\alpha-1}(56n+49)) \equiv 0$ (mod 7) and
\[{{\overline{p}}_{3}}\big({{7}^{2\alpha +1}}(7n+3)\big)\equiv {{\overline{p}}_{3}}\big({{7}^{2\alpha +1}}(7n+5)\big)\equiv {{\overline{p}}_{3}}\big({{7}^{2\alpha +1}}(7n+6)\big)\equiv 0 \pmod{7}.\]

% -----------------------------------------------------------
\section{Congruences  Modulo Powers of 2}

\begin{thm}\label{mod16}
We have
\begin{displaymath}
{{\overline{p}}_{k}}(n)\equiv \left\{ \begin{array}{ll}
   -2k, &\textrm{if $n$ is an even square}; \\
 2k, &\textrm{if $n$ is an odd square}; \\
  2k(k+1), &\textrm{if $n$ is twice a square}; \\
  0, & \textrm{else}
\end{array} \right. \pmod{{{2}^{m (k)}}}.
\end{displaymath}
where $m(k)=4$ when $k\equiv 0,3$ \text{\rm{(mod $4$)}} and $m(k)=3$ otherwise .
\end{thm}
 Comparing with (\ref{mod2}),  it is easy to see that this result has the advantage when $k$ is odd.

To prove this theorem, we need a simple lemma as follows.

\begin{lem}\label{2square} Let $\chi (n)$ equal to 1 if $n$ is a square and 0 otherwise, and let $r_{2}^{+}(n)$ be the number of representations of $n$ as the form ${{i}^{2}}+{{j}^{2}}$ where both $i$ and $j$ are positive integers. Then for $n\ge 1$,
\[r_{2}^{+}(n)=\frac{1}{4}{{r}_{2}}(n)-\chi (n).\]
\end{lem}

\begin{proof}
If $n$ is not a square, then  among the four representations  $n={{(\pm x)}^{2}}+{{(\pm y)}^{2}}$ $(xy\ne 0)$ counted by $r_{2}(n)$, there is exactly one representation $|x{{|}^{2}}+|y{{|}^{2}}$ counted by $r_{2}^{+}(n)$. Hence $r_{2}^{+}(n)={{r}_{2}}(n)/4.$

If $n={{k}^{2}}$ is a square, then we should first eliminate the four representations ${{k}^{2}}+{{0}^{2}}$, ${{(-k)}^{2}}+{{0}^{2}}$, ${{0}^{2}}+{{k}^{2}}$, ${{0}^{2}}+{{(-k)}^{2}}$ from the representations enumerated by ${{r}_{2}}(n)$. The result follows similarly by the arguments above.
\end{proof}

\begin{proof}[Proof of Theorem \ref{mod16}]
Observe that $\varphi (q)=1+2S(q)$ where $S(q)=\sum\nolimits_{n\ge 1}{{{q}^{{{n}^{2}}}}}$.  We have
\begin{displaymath}
\begin{split}
  \sum\limits_{n=0}^{\infty }{{{\overline{p}}_{k}}(n){{(-q)}^{n}}} & =\frac{1}{\varphi {{(q)}^{k}}}=\frac{1}{{{(1+2S(q))}^{k}}} \\
  & \equiv {{\Big(1-2S(q)+4S{{(q)}^{2}}-8S{{(q)}^{3}}\Big)}^{k}} \\
 & \equiv 1-2kS(q)+2k(k+1)S{{(q)}^{2}}-\frac{4}{3}k(k+1)(k+2)S{{(q)}^{3}} \pmod{16}.
\end{split}
\end{displaymath}
If $k\equiv 0,3$ (mod $4$), then  $\frac{4}{3}k(k+1)(k+2)\equiv 0$ (mod $16$). If $k\equiv 1,2$ (mod $4$), then $\frac{4}{3}k(k+1)(k+2)\equiv 0$ (mod $8$). We have
\begin{equation}\label{pkmid}
\sum\limits_{n\ge 0}{{{\overline{p}}_{k}}(n){{(-q)}^{n}}}\equiv 1-2kS(q)+2k(k+1)S{{(q)}^{2}} \pmod{{{2}^{m(k)}}}.
\end{equation}

For $n\ge 1$, by the definition of $S(q)$ and $r_{2}^{+}(n)$ we obtain
\[{{(-1)}^{n}}{{\overline{p}}_{k}}(n)\equiv 2k(k+1)r_{2}^{+}(n)-2k\chi (n) \pmod{{{2}^{m(k)}}}.\]
By Lemma \ref{2square} we can write
\begin{equation}\label{p3n16}
(-1)^{n}{{\overline{p}}_{k}}(n)\equiv \frac{1}{2}k(k+1){{r}_{2}}(n)-2k(k+2)\chi (n) \pmod{{{2}^{m(k)}}}.
\end{equation}
Let $n={{2}^{\alpha }}(2{{n}_{0}}+1)$, $\alpha ,{{n}_{0}}\ge 0$. From \cite[Theorem 3.2.1]{Bruce} we see that
\[{{r}_{2}}(n)=4\sum\limits_{\begin{smallmatrix}
 d|n \\
 d \,\mathrm{odd}
\end{smallmatrix}}{{{(-1)}^{(d-1)/2}}}\equiv 4\sum\limits_{\begin{smallmatrix}
 d|n \\
 d \,\textrm{odd}
\end{smallmatrix}}{1} \equiv 4d(2{{n}_{0}}+1) \pmod{8},\]
here $d(2{n}_{0}+1)$ denotes the number of positive divisors of $2{n}_{0}+1$. It is well-known that  $d(2{{n}_{0}}+1)$ is odd if and only if  $2{{n}_{0}}+1$ is a square.  Hence we have
\begin{equation}\label{r2n}
{{r}_{2}}(n)\equiv \left\{ \begin{array}{ll}
  4, & \textrm{$n$ is a square or twice a square}\\
  0, & \textrm{else}
\end{array}\right. \pmod{8}.
\end{equation}
Note that if $k\equiv 0,3$ (mod $4$), then $\frac{1}{2}k(k+1)\equiv 0$ (mod $2$). Combining  (\ref{p3n16}) with (\ref{r2n}), we get the desired result.
\end{proof}

Since the number of integers not exceeding $x$  which are a square or twice a square is less than $2\sqrt{x}$, we have
 \begin{cor}
Given any  integer $k\equiv 3$ \text{\rm{(mod $4$)}}, ${{\overline{p}}_{k}}(n)$ is divisible by $16$ for almost all positive integers $n$.
\end{cor}

%------------------------------------------------------------------------------------------------------------
In particular, for  $k=3$ , we have
\begin{equation}\label{p3mod16}
{{\overline{p}}_{3}}(n)\equiv \left\{ \begin{array}{ll}
   10, &\textrm{if $n$ is an even square}; \\
 6, &\textrm{if $n$ is an odd square}; \\
  8, &\textrm{if $n$ is twice a square}; \\
  0, & \textrm{else}
\end{array} \right. \pmod{16}.
\end{equation}
This leads to the following corollary.
\begin{cor}
For any nonnegative integer $n$, we have
\[{{\overline{p}}_{3}}(8n+1)\equiv 0 \pmod{2}, \quad {{\overline{p}}_{3}}(8n+2)\equiv 0 \pmod{8},\]
\[{{\overline{p}}_{3}}(8n+3)\equiv 0 \pmod{16}, \quad {{\overline{p}}_{3}}(8n+4)\equiv 0 \pmod{2},\]
\[{{\overline{p}}_{3}}(8n+5)\equiv 0 \pmod{16}, \quad {{\overline{p}}_{3}}(8n+6)\equiv 0 \pmod{16}.\]
\end{cor}
\begin{proof}
It follows directly from (\ref{p3mod16}) that ${{\overline{p}}_{3}}(8n+1)\equiv {{\overline{p}}_{3}}(8n+4)\equiv 0$ (mod 2).
Since any square is congruent to 0, 1 or 4 modulo 8, we know that $8n+2$ can not be a square. Hence by (\ref{p3mod16}), we deduce that
${{\overline{p}}_{3}}(8n+2)\equiv 0$  (mod 8).
Similarly, for $r \in \{3,5,6\}$, $8n+r$ cannot be a square or twice a square. Hence ${{\overline{p}}_{3}}(8n+r)\equiv 0$ (mod 16).
\end{proof}
Since
\[{{\overline{p}}_{3}}(1)=6,\quad{{\overline{p}}_{3}}(2)=24, \quad {{\overline{p}}_{3}}(3)=80={{2}^{4}}\cdot 5,\quad{{\overline{p}}_{3}}(4)=234=2\cdot 117,\]
\[{{\overline{p}}_{3}}(5)=624={{2}^{4}}\cdot 39,\quad {{\overline{p}}_{3}}(6)=1552={{2}^{4}}\cdot 97,\]
we see that all the modulus in this corollary cannot be replaced by higher powers of 2.

We have two deep Ramanujan-type congruences for ${{\overline{p}}_{3}}(n)$  modulo 32 and 64. This will be the next topic of this section. Before our discussion, we need some lemmas for preparation.

 \begin{lem}\label{keyid1}
 (Cf. \cite[Lemma 1]{Hirschhorn05}.) The following identities hold:
\[\varphi (q)=\varphi ({{q}^{4}})+2q\psi ({{q}^{8}}), \quad \varphi {{(q)}^{2}}=\varphi {{({{q}^{2}})}^{2}}+4q\psi {{({{q}^{4}})}^{2}},\]
\[\varphi (q)\varphi (-q)=\varphi {{(-{{q}^{2}})}^{2}}, \quad \psi {{(q)}^{2}}=\varphi (q)\psi ({{q}^{2}}).\]
\end{lem}
\begin{lem}\label{keyid2}
(Cf. \cite[Entry 25(viii)]{notebook}.)We have
\[\varphi {{(q)}^{4}}-\varphi {{(-q)}^{4}}=16q\psi {{({{q}^{2}})}^{4}}.\]
\end{lem}

 \begin{lem}\label{keyid3}
 (Cf. \cite[Lemma 2.1]{Chen}.) We have
\[\frac{1}{\varphi (-q)}=\frac{1}{\varphi {{(-{{q}^{4}})}^{4}}}\Big(\varphi {{({{q}^{4}})}^{3}}+2q\varphi {{({{q}^{4}})}^{2}}\psi ({{q}^{8}})+4{{q}^{2}}\varphi ({{q}^{4}})\psi {{({{q}^{8}})}^{2}}+8{{q}^{3}}\psi {{({{q}^{8}})}^{3}}\Big).\]
\end{lem}

\begin{thm}\label{p38n7}
 For any integer $ n \ge 0$, we have
\[{{\overline{p}}_{3}}(16n+14)\equiv 0 \pmod{32}.\]
\end{thm}

\begin{proof}
Throughout this section, we will denote $a=\varphi ({{q}^{4}})$ and $b=\psi ({{q}^{8}})$. By Lemma \ref{keyid1}, we have $\varphi (q)=a+2qb$.  By Lemma \ref{keyid3}, we have
\begin{equation}\label{p3mod2expan}
\sum\limits_{n\ge 0}{{{\overline{p}}_{3}}(n){{q}^{n}}}=\frac{1}{\varphi {{(-q)}^{3}}}=\frac{1}{\varphi {{(-{{q}^{4}})}^{12}}}{{\big({{a}^{3}}+2q{{a}^{2}}b+4{{q}^{2}}a{{b}^{2}}+8{{q}^{3}}{{b}^{3}}\big)}^{3}}.
\end{equation}
If we collect all the terms of the form ${{q}^{4n+2}}$ in the expansion of $({{a}^{3}}+2q{{a}^{2}}b+4{{q}^{2}}a{{b}^{2}}$\\
$+8{{q}^{3}}{{b}^{3}})^{3}$, we get \[{{q}^{2}}(24{{a}^{7}}{{b}^{2}}+640{{a}^{3}}{{b}^{6}}{{q}^{4}})\equiv 24{{q}^{2}}{{a}^{7}}{{b}^{2}} \pmod{{{2}^{7}}}.\]
Collecting all the terms of the form ${{q}^{4n+2}}$ on both sides of  (\ref{p3mod2expan}), dividing by $q^{2}$ and replacing $q^{4}$ by $q$, we obtain
\begin{equation}\label{term11}
\sum\limits_{n\ge 0}{{{\overline{p}}_{3}}(4n+2){{q}^{n}}}\equiv 24\frac{\varphi {{(q)}^{7}}\psi {{({{q}^{2}})}^{2}}}{\varphi {{(-q)}^{12}}} \pmod{{{2}^{7}}}.
\end{equation}
By Lemma \ref{keyid2}, we have $\varphi {{(-q)}^{4}}\equiv \varphi {{(q)}^{4}}$ (mod ${{2}^{4}}$). Hence from (\ref{term11}) we deduce that
\begin{equation}\label{term12}
\sum\limits_{n\ge 0}{{{\overline{p}}_{3}}(4n+2){{q}^{n}}}\equiv 24\frac{\varphi {{(q)}^{7}}\psi {{({{q}^{2}})}^{2}}}{\varphi {{(q)}^{12}}}=24\frac{\psi {{({{q}^{2}})}^{2}}}{\varphi {{(q)}^{5}}} \pmod{{{2}^{7}}}.
\end{equation}

Because $(a+2qb)^{6}\equiv {{a}^{6}}$ (mod $4$), we get $\varphi {{(q)}^{6}}\equiv \varphi {{({{q}^{4}})}^{6}}$ (mod $4$). Thus from (\ref{term12}) we obtain
\[\sum\limits_{n\ge 0}{{{\overline{p}}_{3}}(4n+2){{q}^{n}}}\equiv 24\frac{\psi {{({{q}^{2}})}^{2}}}{\varphi {{(q)}^{6}}}\varphi (q)\equiv 24\frac{\psi {{({{q}^{2}})}^{2}}}{\varphi {{({{q}^{4}})}^{6}}}\Big(\varphi ({{q}^{4}})+2q\psi ({{q}^{8}})\Big) \pmod{{{2}^{5}}}.\]
Collecting all the terms of the form ${{q}^{2n+1}}$ on both sides, dividing by $q$ and replacing ${{q}^{2}}$ by $q$, we obtain
\[\sum\limits_{n\ge 0}{{{\overline{p}}_{3}}(8n+6){{q}^{n}}} \equiv 16\frac{{\psi (q)}^{2}\psi ({{q}^{4}})}{\varphi {{({{q}^{2}})}^{6}}} \equiv 16\frac{\psi ({{q}^{2}})\psi ({{q}^{4}})}{\varphi {{({{q}^{2}})}^{6}}} \pmod{32}.\]
Note that the terms of the form ${{q}^{2n+1}}$ do not appear in the right hand side, we deduce that ${{\overline{p}}_{3}}\big(8(2n+1)+6\big)\equiv 0$ (mod  $32$), which is the desired result we want.
\end{proof}

\begin{rem}
The modulus 32 cannot be replaced by 64 since for $n=0$, we have ${{\overline{p}}_{3}}(14)=535008={{2}^{5}}\cdot 3\cdot 5573$.
\end{rem}

\begin{thm}
 For any integer $ n \ge 0$, we have
\[{{\overline{p}}_{3}}(8n+7)\equiv 0 \pmod{64}.\]
\end{thm}

\begin{proof}
We start from (\ref{p3mod2expan}). If we collect all the terms of the form ${{q}^{4n+3}}$ in the expansion of ${{({{a}^{3}}+2q{{a}^{2}}b+4{{q}^{2}}a{{b}^{2}}+8{{q}^{3}}{{b}^{3}})}^{3}}$,
we get
\[{{q}^{3}}(80{{a}^{6}}{{b}^{3}}+768{{a}^{2}}{{b}^{7}}{{q}^{4}})\equiv 80{{q}^{3}}{{a}^{6}}{{b}^{3}} \pmod{{{2}^{8}}}.\]
Collecting all  the terms of the form ${{q}^{4n+3}}$ on both sides of  (\ref{p3mod2expan}),  then dividing  by ${{q}^{3}}$ and replacing ${{q}^{4}}$ by $q$, we  obtain
\begin{equation}\label{term21}
\sum\limits_{n\ge 0}{{{\overline{p}}_{3}}(4n+3){{q}^{n}}}\equiv 80\frac{\varphi {{(q)}^{6}}\psi {{({{q}^{2}})}^{3}}}{\varphi {{(-q)}^{12}}} \pmod{{{2}^{8}}}.
\end{equation}
In the proof of Theorem \ref{p38n7} we have seen that  $\varphi {{(q)}^{6}}\equiv \varphi {{({{q}^{4}})}^{6}}$ (mod $4$). Hence $\varphi {{(-q)}^{6}}\equiv \varphi {{({{q}^{4}})}^{6}}\equiv \varphi {{(q)}^{6}}$ (mod 4). We deduce from (\ref{term21}) that
\[\sum\limits_{n\ge 0}{{{\overline{p}}_{3}}(4n+3){{q}^{n}}}\equiv 80\frac{\psi {{({{q}^{2}})}^{3}}}{\varphi {{(q)}^{6}}}\equiv 80\frac{\psi {{({{q}^{2}})}^{3}}}{\varphi {{({{q}^{4}})}^{6}}} \pmod{{{2}^{6}}}.\]
Since the terms of the form ${{q}^{2n+1}}$ do not appear in the right hand side, we deduce that ${{\overline{p}}_{3}}(4(2n+1)+3)\equiv 0$ (mod $64$). This completes the proof.
\end{proof}
\begin{rem}
The modulus 64 in this theorem  cannot be replaced by 128 as we have ${{\overline{p}}_{3}}(7)=3648={{2}^{6}}\cdot 3\cdot 19$.
\end{rem}

% ------------------------------------------------------------------------------------------------------------------------------------------------------
\section{Congruences Modulo 9}
 First we state two lemmas which will play a key role in our argument.
\begin{lem}\label{identity1}
(Cf. \cite[Lemmas 2.6 and 2.7]{ped}.)
The following identities hold:
\[\varphi (q)=\varphi ({{q}^{9}})+2qB(-{{q}^{3}}),\]
\[\varphi {{({{q}^{3}})}^{3}}+8qB{{(-q)}^{3}}=\frac{\varphi {{(q)}^{4}}}{\varphi ({{q}^{3}})},\]
where
\[B(q)=\frac{{{(q;q)}_{\infty }}({{q}^{6}};{{q}^{6}})_{\infty }^{2}}{{{({{q}^{2}};{{q}^{2}})}_{\infty }}{{({{q}^{3}};{{q}^{3}})}_{\infty }}}.\]
\end{lem}

\begin{lem}\label{identity2}
(Cf. \cite[Lemma 2.1]{Chen}.) We have
\[\frac{1}{\varphi (q)}=\frac{\varphi ({{q}^{9}})}{\varphi {{({{q}^{3}})}^{4}}}\Big(\varphi {{({{q}^{9}})}^{2}}-2q\varphi ({{q}^{9}})B({{-q}^{3}})+4{{q}^{2}}B{{({{-q}^{3}})}^{2}}\Big).\]
\end{lem}
The following lemma will also be used frequently.
\begin{lem}\label{basic}
Let $p$ be a prime. We have
\begin{displaymath}
(q;q)_{\infty }^{p}  \equiv ({{q}^{p}};{{q}^{p}})_{\infty }  \pmod  {p}, \quad \varphi (q)^{p}  \equiv \varphi (q^{p}) \pmod{p}.
\end{displaymath}
\end{lem}
\begin{proof}
By the binomial theorem, we have $(1-q)^{p} \equiv 1-q^{p}$ (mod $p$). This proves the first congruence.  Since $\varphi (q)=(-q;-q)_{\infty }^{2}/{{({{q}^{2}};{{q}^{2}})}_{\infty }}$, the second congruence relation follows from the first one.
\end{proof}
For convenience, we will let  $s=\varphi ({{q}^{9}})$ and $t=2B(-{{q}^{3}})$ throughout this section. From Lemmas \ref{identity1} and \ref{identity2}, we have
\begin{equation}\label{stone}
{{s}^{3}}+{{q}^{3}}{{t}^{3}}=\frac{\varphi {{({{q}^{3}})}^{4}}}{\varphi ({{q}^{9}})},
\end{equation}
\begin{equation}\label{sttwo}
\frac{1}{\varphi (q)}=\frac{\varphi ({{q}^{9}})}{\varphi {{({{q}^{3}})}^{4}}}({{s}^{2}}-qst+{{q}^{2}}{{t}^{2}}).
\end{equation}

We are going to establish an infinite family of congruences modulo $9\cdot 2^4$.
\begin{thm}\label{p3thm1}
We have
\begin{displaymath}
\begin{split}
\sum\limits_{n\ge 0}{\overline{{{p}}}_{3}(3n){{(-q)}^{n}}}&\equiv \frac{\varphi ({{q}^{3}})}{\varphi {{(q)}^{4}}} \pmod{9\cdot 2^3},\\
\sum\limits_{n\ge 0}{\overline{{{p}}}_{3}(3n+1){{(-q)}^{n}}}&\equiv 6\frac{\varphi {{({{q}^{3}})}^{4}}}{\varphi {{(q)}^{8}}}B(-q) \pmod{9\cdot 2^4},\\
\sum\limits_{n\ge 0}{\overline{{{p}}}_{3}(3n+2){{(-q)}^{n}}}&\equiv 24\frac{\varphi {{({{q}^{3}})}^{3}}}{\varphi {{(q)}^{8}}}B{{(-q)}^{2}} \pmod{9\cdot 2^5}.
\end{split}
\end{displaymath}
\end{thm}

 \begin{proof}
By (\ref{sttwo}) we have
\begin{equation}\label{pterm}
\sum\limits_{n\ge 0}{{{\overline{p}}_{3}}(n){{(-q)}^{n}}}=\frac{1}{\varphi {{(q)}^{3}}}=\frac{\varphi {{({{q}^{9}})}^{3}}}{\varphi {{({{q}^{3}})}^{12}}}{{({{s}^{2}}-qst+{{q}^{2}}{{t}^{2}})}^{3}}.
\end{equation}
Note that
\begin{displaymath}
{{({{s}^{2}}-qst+{{q}^{2}}{{t}^{2}})}^{3}}=({{s}^{6}}-7{{q}^{3}}{{s}^{3}}{{t}^{3}}+{{q}^{6}}{{t}^{6}})+q(-3{{s}^{5}}t+6{{q}^{3}}{{s}^{2}}{{t}^{4}})+{{q}^{2}}(6{{s}^{4}}{{t}^{2}}-3{{q}^{3}}s{{t}^{5}}).
\end{displaymath}
Since $t=2B(-q^3)$, collecting all the terms of the form $q^{3n}$ on both sides of (\ref{pterm}) and applying (\ref{stone}),  we obtain
\begin{displaymath}
\begin{split}
   \sum\limits_{n\ge 0}{{{\overline{p}}_{3}}(3n){{(-q)}^{3n}}}&=\frac{\varphi {{({{q}^{9}})}^{3}}}{\varphi {{({{q}^{3}})}^{12}}}({{s}^{6}}-7{{q}^{3}}{{s}^{3}}{{t}^{3}}+{{q}^{6}}{{t}^{6}}) \\
 & \equiv \frac{\varphi {{({{q}^{9}})}^{3}}}{\varphi {{({{q}^{3}})}^{12}}}{{({{s}^{3}}+{{q}^{3}}{{t}^{3}})}^{2}} \\
 & \equiv \frac{\varphi ({{q}^{9}})}{\varphi {{({{q}^{3}})}^{4}}} \pmod{9\cdot 2^3}.
\end{split}
\end{displaymath}
Replacing ${{q}^{3}}$ by $q$, we get the first congruence identity.

Similarly, we have
\begin{displaymath}
\begin{split}
   \sum\limits_{n\ge 0}{{{\overline{p}}_{3}}(3n+1){{(-q)}^{3n+1}}}&=q\frac{\varphi {{({{q}^{9}})}^{3}}}{\varphi {{({{q}^{3}})}^{12}}}(-3{{s}^{5}}t+6{{q}^{3}}{{s}^{2}}{{t}^{4}}) \\
 & \equiv -3q\frac{\varphi {{({{q}^{9}})}^{3}}}{\varphi {{({{q}^{3}})}^{12}}}{{s}^{2}}t({{s}^{3}}+{{q}^{3}}{{t}^{3}}) \\
 & \equiv -6q\frac{\varphi {{({{q}^{9}})}^{4}}}{\varphi {{({{q}^{3}})}^{8}}}B(-{{q}^{3}}) \pmod{9 \cdot 2^4}.
\end{split}
\end{displaymath}
Dividing both sides by $-q$ and replacing ${{q}^{3}}$ by $q$, we get the second congruence identity.

In the same way, we have
\begin{displaymath}
\begin{split}
   \sum\limits_{n\ge 0}{{{\overline{p}}_{3}}(3n+2){{(-q)}^{3n+2}}}  & ={{q}^{2}}\frac{\varphi {{({{q}^{9}})}^{3}}}{\varphi {{({{q}^{3}})}^{12}}}(6{{s}^{4}}{{t}^{2}}-3s{{t}^{5}}{{q}^{3}}) \\
 & \equiv 6{{q}^{2}}\frac{\varphi {{({{q}^{9}})}^{3}}}{\varphi {{({{q}^{3}})}^{12}}}s{{t}^{2}}({{s}^{3}}+{{q}^{3}}{{t}^{3}}) \\
 & \equiv 24{{q}^{2}}\frac{\varphi {{({{q}^{9}})}^{3}}}{\varphi {{({{q}^{3}})}^{8}}}B{{(-{{q}^{3}})}^{2}} \pmod{9\cdot 2^5}.
\end{split}
\end{displaymath}
Dividing both sides by ${{q}^{2}}$ and replacing ${{q}^{3}}$ by $q$, we get the third congruence identity.
\end{proof}
The following corollary follows immediately.
\begin{cor}
For any nonnegative integer $n$, we have
\[{{\overline{p}}_{3}}(3n+1)\equiv 0 \pmod{6}, \quad {{\overline{p}}_{3}}(3n+2)\equiv 0 \pmod{24}.\]
\end{cor}

\begin{thm}\label{p3thm2}
For any integer $\alpha \ge 1$, we have
\begin{displaymath}
\begin{split}
\sum\limits_{n\ge 0}{{{\overline{p}}_{3}}({{3}^{2\alpha }}n){{(-q)}^{n}}}  &\equiv \frac{\varphi {{({{q}^{3}})}^{4}}}{\varphi {{(q)}^{7}}} \pmod{9\cdot 2^3}, \\
\sum\limits_{n\ge 0}{{{\overline{p}}_{3}}({{3}^{2\alpha +1}}n){{(-q)}^{n}}} &\equiv \frac{\varphi {{({{q}^{3}})}^{5}}}{\varphi {{(q)}^{8}}} \pmod{9\cdot 2^3}.
\end{split}
\end{displaymath}
\end{thm}
\begin{proof}
We proceed our proof by induction on $\alpha $.

By (\ref{sttwo}) and the first identity in Theorem \ref{p3thm1}, we have
\[\sum\limits_{n\ge 0}{{{\overline{p}}_{3}}(3n){{(-q)}^{n}}} \equiv \frac{\varphi ({{q}^{3}})}{\varphi ({{q}})^{4}} = \frac{\varphi {{({{q}^{9}})}^{4}}}{\varphi {{({{q}^{3}})}^{15}}}{{({{s}^{2}}-qst+{{q}^{2}}{{t}^{2}})}^{4}} \pmod{9\cdot 2^3}.\]
Collecting all the terms of the form ${{q}^{3n}}$ in the expansion of ${{({{s}^{2}}-qst+{{q}^{2}}{{t}^{2}})}^{4}}$ and reducing modulo 9, we get
\[{{s}^{8}}-16{{q}^{3}}{{s}^{5}}{{t}^{3}}+10{{q}^{6}}{{s}^{2}}{{t}^{6}}\equiv {{s}^{2}}{{({{s}^{3}}+{{q}^{3}}{{t}^{3}})}^{2}} \pmod{9\cdot 2^3}.\]
Hence we have
\[\sum\limits_{n\ge 0}{{{\overline{p}}_{3}}(9n){{(-q)}^{3n}}}\equiv \frac{\varphi {{({{q}^{9}})}^{4}}}{\varphi {{({{q}^{3}})}^{15}}}{{s}^{2}}{{({{s}^{3}}+{{q}^{3}}{{t}^{3}})}^{2}}=\frac{\varphi {{({{q}^{9}})}^{4}}}{\varphi {{({{q}^{3}})}^{7}}} \pmod{9\cdot 2^3}.\]
Replacing ${{q}^{3}}$ by $q$ we proved the first identity for $\alpha =1$.

 Suppose
\begin{equation}\label{a66}
\sum\limits_{n\ge 0}{{{\overline{p}}_{3}}({{3}^{2\alpha }}n){{(-q)}^{n}}}\equiv \frac{\varphi {{({{q}^{3}})}^{4}}}{\varphi {{(q)}^{7}}} \pmod{9\cdot 2^3}.
\end{equation}
Then by (\ref{sttwo}) we have
\[\frac{\varphi {{({{q}^{3}})}^{4}}}{\varphi {{(q)}^{7}}}=\frac{\varphi {{({{q}^{9}})}^{7}}}{\varphi {{({{q}^{3}})}^{24}}}{{({{s}^{2}}-qst+{{q}^{2}}{{t}^{2}})}^{7}}.\]
Collecting those terms of the form $q^{3n}$ in the expansion of ${{({{s}^{2}}-qst+{{q}^{2}}{{t}^{2}})}^{7}}$ and reducing modulo 9, we get
\begin{displaymath}
\begin{split}
  & {{s}^{14}}-77{{q}^{3}}{{s}^{11}}{{t}^{3}}+357{{q}^{6}}{{s}^{8}}{{t}^{6}}-266{{q}^{9}}{{s}^{5}}{{t}^{9}}+28{{q}^{12}}{{s}^{2}}{{t}^{12}} \\
 & \equiv {{s}^{14}}+4{{q}^{3}}{{s}^{11}}{{t}^{3}}+6{{q}^{6}}{{s}^{8}}{{t}^{6}}+4{{q}^{9}}{{s}^{5}}{{t}^{9}}+{{q}^{12}}{{s}^{2}}{{t}^{12}} \\
 & ={{s}^{2}}{{({{s}^{3}}+{{q}^{3}}{{t}^{3}})}^{4}} \pmod{9\cdot 2^3}.
\end{split}
\end{displaymath}
If we collect all the terms of the form ${{q}^{3n}}$ in (\ref{a66}), we get
\[\sum\limits_{n\ge 0}{{{\overline{p}}_{3}}({{3}^{2\alpha +1}}n){{(-q)}^{3n}}}\equiv \frac{\varphi {{({{q}^{9}})}^{7}}}{\varphi {{({{q}^{3}})}^{24}}}{{s}^{2}}{{({{s}^{3}}+{{q}^{3}}{{t}^{3}})}^{4}}\equiv \frac{\varphi {{({{q}^{9}})}^{5}}}{\varphi {{({{q}^{3}})}^{8}}} \pmod{9\cdot 2^3}.\]
Replacing ${{q}^{3}}$ by $q$, we obtain
\begin{equation}\label{induction}
\sum\limits_{n\ge 0}{{{\overline{p}}_{3}}({{3}^{2\alpha +1}}n){{(-q)}^{n}}}\equiv \frac{\varphi {{({{q}^{3}})}^{5}}}{\varphi {{(q)}^{8}}} \pmod{9\cdot 2^3}.
\end{equation}

Now suppose we have (\ref{induction}),  by (\ref{sttwo}) we deduce that
\[\frac{\varphi {{({{q}^{3}})}^{5}}}{\varphi {{(q)}^{8}}}=\frac{\varphi {{({{q}^{9}})}^{8}}}{\varphi {{({{q}^{3}})}^{27}}}{{({{s}^{2}}-qst+{{q}^{2}}{{t}^{2}})}^{8}}.\]
Collecting all the terms of the form ${{q}^{3n}}$ in the expansion of ${{({{s}^{2}}-qst+{{q}^{2}}{{t}^{2}})}^{8}}$ and reducing modulo 9, we get
\begin{displaymath}
\begin{split}
  & {{s}^{16}}-112{{q}^{3}}{{s}^{13}}{{t}^{3}}+784{{q}^{6}}{{s}^{10}}{{t}^{6}}-1016{{q}^{9}}{{s}^{7}}{{t}^{9}}+266{{q}^{12}}{{s}^{4}}{{t}^{12}}-8{{q}^{15}}s{{t}^{15}} \\
 & \equiv s({{s}^{15}}+5{{q}^{3}}{{s}^{12}}{{t}^{3}}+10{{q}^{6}}{{s}^{9}}{{t}^{6}}+10{{q}^{9}}{{s}^{6}}{{t}^{9}}+5{{q}^{12}}{{s}^{3}}{{t}^{12}}+{{q}^{15}}{{t}^{15}}) \\
 & =s{{({{s}^{3}}+{{q}^{3}}{{t}^{3}})}^{5}} \pmod{9\cdot 2^3}.
\end{split}
\end{displaymath}
If we collect all the terms of the form ${{q}^{3n}}$ in (\ref{induction}), we get
\[\sum\limits_{n\ge 0}{{{\overline{p}}_{3}}({{3}^{2\alpha +2}}n){{(-q)}^{3n}}}\equiv \frac{\varphi {{({{q}^{9}})}^{8}}}{\varphi {{({{q}^{3}})}^{27}}}s{{({{s}^{3}}+{{q}^{3}}{{t}^{3}})}^{5}}\equiv \frac{\varphi {{({{q}^{9}})}^{4}}}{\varphi {{({{q}^{3}})}^{7}}} \pmod{9\cdot 2^3}.\]
Replacing ${{q}^{3}}$ by $q$, we obtain
\[\sum\limits_{n\ge 0}{{{\overline{p}}_{3}}({{3}^{2\alpha +2}}n){{(-q)}^{n}}}\equiv \frac{\varphi {{({{q}^{3}})}^{4}}}{\varphi {{(q)}^{7}}} \pmod{9\cdot 2^3}.\]
By inducting on $\alpha $, we complete our proof.
\end{proof}

\begin{thm}\label{p3mod9}
For any integers $\alpha \ge 1$ and $n\ge 0$, we have
\[{{\overline{p}}_{3}}\big({{3}^{2\alpha +1}}(3n+2)\big)\equiv 0 \pmod{144}.\]
\end{thm}
\begin{proof}
By (\ref{sttwo}) we have
\[\frac{\varphi {{({{q}^{3}})}^{5}}}{\varphi {{(q)}^{8}}}=\frac{\varphi {{({{q}^{9}})}^{8}}}{\varphi {{({{q}^{3}})}^{27}}}{{({{s}^{2}}-qst+{{q}^{2}}{{t}^{2}})}^{8}}.\]
If we collect all the terms of the form ${{q}^{3n+2}}$ in the expansion of ${{({{s}^{2}}-qst+{{q}^{2}}{{t}^{2}})}^{8}}$, we get \[9{{q}^{2}}(4{{s}^{14}}{{t}^{2}}-56{{q}^{3}}{{s}^{11}}{{t}^{5}}+123{{q}^{6}}{{s}^{8}}{{t}^{8}}-56{{q}^{9}}{{s}^{5}}{{t}^{11}}+4{{q}^{12}}{{s}^{2}}{{t}^{14}}).\]
It is obvious that all the coefficients are divisible by 9. Thus all the coefficients of the terms ${{q}^{3n+2}}$  ($n=0,1,2,\cdots $) in  both sides of (\ref{induction}) are divisible by 9, which implies ${{\overline{p}}_{3}}({{3}^{2\alpha +1}}(3n+2))\equiv 0$ (mod $9$).

For any integer $x$, we have ${{x}^{2}}\equiv 0, 1, 4, 7$ (mod $9$). Hence  ${{3}^{2\alpha +1}}(3n+2)={{9}^{\alpha }}(9n+6)$ cannot be a square. Moreover, since $2{{x}^{2}}\equiv 0, 2, 5, 8$ (mod $9$), we know $9n+6$ cannot be twice a square. Thus ${{9}^{\alpha }}(9n+6)$ cannot be twice a square. By (\ref{p3mod16}) we know ${{\overline{p}}_{3}}({{3}^{2\alpha +1}}(3n+2))\equiv 0$ (mod $16$).
Combining this with ${{\overline{p}}_{3}}({{3}^{2\alpha +1}}(3n+2))\equiv 0$ (mod $9$), we complete our proof.
\end{proof}

 \begin{cor}
${{\overline{p}}_{3}}(n)$ is divisible by $144$ for at least $\frac{1}{72}$ of the time.
\end{cor}
\begin{proof}
The arithmetic sequences $\big\{{{3}^{2\alpha +1}}(3n+2):n\ge 0\big\}$ ($\alpha =1,2,\cdots $) mentioned in Theorem \ref{p3mod9}
 are disjoint for different $\alpha$, and they account for
\[\frac{1}{{{3}^{4}}}+\frac{1}{{{3}^{6}}}+\cdots =\frac{1}{72}\]
of all positive integers.
\end{proof}

\section{Congruences Modulo 7 and 11}
\begin{thm}\label{p3r3}
For any integer $n\ge 1$, we have
\[{{\overline{p}}_{3}}(7n)\equiv {{(-1)}^{n}}{{r}_{3}}(n) \pmod{7}.\]
\end{thm}
To prove this theorem, we need the following lemma.
\begin{lem}\label{r48modp}
(Cf. \cite[Lemma 12]{Wang}.)
For any prime $p\ge 3$, we have
 \[{{r}_{4}}(pn)\equiv {{r}_{4}}(n) \pmod{p}, \quad {{r}_{8}}(pn)\equiv {{r}_{8}}(n)  \pmod{p}.\]
\end{lem}

\begin{proof}[Proof of Theorem \ref{p3r3}]
We have
\[\varphi {{(q)}^{7}}\sum\limits_{n\ge 0}{{{\overline{p}}_{3}}(n){{(-q)}^{n}}}=\varphi {{(q)}^{4}}=\sum\limits_{n\ge 0}{{{r}_{4}}(n){{q}^{n}}}.\]
By Lemma \ref{basic}, we have $\varphi {{(q)}^{7}}\equiv \varphi ({{q}^{7}})$ (mod $7$). Hence
\[\varphi ({{q}^{7}})\sum\limits_{n\ge 0}{{{\overline{p}}_{3}}(n){{(-q)}^{n}}}\equiv \sum\limits_{n\ge 0}{{{r}_{4}}(n){{q}^{n}}} \pmod{7}.\]
Collecting all the terms of the form ${{q}^{7n}}$ on both sides of this identity, then replacing ${{q}^{7}}$  by $q$ and applying Lemma \ref{r48modp} with $p=7$, we obtain
\[\varphi (q)\sum\limits_{n\ge 0}{{{\overline{p}}_{3}}(7n){{(-q)}^{n}}}\equiv \sum\limits_{n\ge 0}{{{r}_{4}}(7n){{q}^{n}}}\equiv \sum\limits_{n\ge 0}{{{r}_{4}}(n){{q}^{n}}}=\varphi {{(q)}^{4}}  \pmod{7}.\]
This implies
\begin{equation}\label{mod7eq}
\sum\limits_{n\ge 0}{{{\overline{p}}_{3}}(7n){{(-q)}^{n}}}\equiv \varphi {{(q)}^{3}}=\sum\limits_{n\ge 0}{{{r}_{3}}(n){{q}^{n}}}  \pmod{7}.
\end{equation}
The theorem follows by comparing the coefficients of ${{q}^{n}}$ on both sides.
\end{proof}

\begin{thm}
For any nonnegative integers $\alpha$ and $n$,  we have
\[{{\overline{p}}_{3}}(7\cdot {{4}^{\alpha +1}}n)\equiv {{(-1)}^{n}}{{\overline{p}}_{3}}(7n) \pmod{7}\]
and
\[{{\overline{p}}_{3}}({{4}^{\alpha }}(56n+49))\equiv 0 \pmod{7}.\]
\end{thm}
\begin{proof}
From (\ref{mod7eq}) and Lemma \ref{keyid1} we see that
\begin{equation}\label{p3expannew}
\begin{split}
&\quad \sum\limits_{n=0}^{\infty}{\overline{p}_{3}(7n)q^{n}}\equiv \varphi(-q)^{3} =(\varphi(q^4)-2q\psi(q^8))^{3} \\
&=\varphi(q^4)^3-6q\varphi(q^4)^2\psi(q^8)+12q^2\psi(q^8)^2\varphi(q^4)-8q^3\psi(q^8)^3 \pmod{7}.
\end{split}
\end{equation}
Extracting the terms of the form $q^{4n}$ on both sides, replacing $q^4$ by $q$, we obtain
\[\sum\limits_{n=0}^{\infty}{\overline{p}_{3}(7\cdot 4n)q^{n}}\equiv \varphi(q)^3  =\sum\limits_{n=0}^{\infty}{(-1)^{n}\overline{p}_{3}(7n)q^{n}} \pmod{7}.\]
Hence we have $\overline{p}_{3}(7\cdot 4n) \equiv (-1)^{n}\overline{p}_{3}(7n)$ (mod 7), from which the first congruence follows.

Extracting the terms of the form $q^{4n+3}$ on both sides of (\ref{p3expannew}), dividing by $q^3$ and replacing $q^4$ by $q$, we obtain
\[\sum\limits_{n=0}^{\infty}{\overline{p}_{3}(7\cdot (4n+3))q^{n}}\equiv -8\psi(q^2)^{3} \pmod{7}.\]
Since the terms of the form $q^{2n+1}$ do not appear on the right hand side, we deduce that $\overline{p}_{3}(7(8n+7)) \equiv 0$ (mod 7). Combining this congruence with the relation ${{\overline{p}}_{3}}(7\cdot {{4}^{\alpha +1}}n)\equiv {{(-1)}^{n}}{{\overline{p}}_{3}}(7n) \pmod{7}$, we proved the second congruence.
\end{proof}

\begin{thm}
Let $p\ge 3$ be a prime, and $N$ be a positive integer which is coprime to $p$. Let $\alpha $ be any nonnegative integer.\\
(1)  If $p\equiv 1$ \text{\rm{(mod $7$)}}, then ${{\overline{p}}_{3}}(7{{p}^{14\alpha +13}}N)\equiv 0$ \text{\rm{(mod $7$)}}. \\
(2) If $p\equiv 2,3,4,5,6$  \text{\rm{(mod $7$)}}, then ${{\overline{p}}_{3}}(7{{p}^{12\alpha +11}}N)\equiv 0$ \text{\rm{(mod $7$)}}. \\
(3) If $p\equiv 6$ \text{\rm{(mod $7$)}}, then ${{\overline{p}}_{3}}(7{{p}^{3}}N)\equiv 0$ \text{\rm{(mod $7$)}}.
\end{thm}
\begin{proof}
From \cite[(4)]{3square} we have
\begin{equation}\label{r3start}
r_{3}(p^2n)+\Big(\frac{-n}{p}\Big)r_3(n)+pr_{3}(n/p^2)=(p+1)r_{3}(n),
\end{equation}
here $(\frac{\cdot}{p})$ is the Legendre symbol and we take $r_{3}(n/p^2)=0$ unless $p^2|n$.
Replacing $n$ by $np$, we see that
\begin{displaymath}
r_{3}(p^3n)+pr_{3}(n/p)=(p+1)r_{3}(pn).
\end{displaymath}

By Theorem \ref{p3r3}, we have
\begin{equation}\label{r3relation}
\overline{p}_{3}(7p^3n)+p\overline{p}_{3}(7n/p)\equiv (p+1)\overline{p}_{3}(7pn) \pmod{7}.
\end{equation}

 From (\ref{r3relation}), by iterating we obtain
\[\overline{p}_{3}(7p^{2k+1}n)\equiv \frac{p(p^k-1)}{p-1}\Big(\overline{p}_{3}(7pn)-\overline{p}_{3}\big(7n/p\big)\Big)+\overline{p}_{3}(7pn) \pmod{7}.\]
Let $n=N$, which is coprime to $p$. We get $\overline{p}_{3}(7n/p)=0$ and thus
\begin{equation}\label{iterate}
\overline{p}_{3}(7p^{2k+1}N) \equiv \Big(\frac{p(p^k-1)}{p-1}+1\Big)\overline{p}_{3}(7pN) \pmod{7}.
\end{equation}

(1) Let $k=7\alpha +6$ in (\ref{iterate}). Since $p\equiv 1$ (mod 7), we have
\[\frac{{{p}^{7\alpha +6}}-1}{p-1}=1+p+\cdots +{{p}^{7\alpha +5}}\equiv -1 \pmod{7}.\]
Hence by (\ref{iterate}) we deduce that  $\overline{p}_{3}(7p^{14\alpha +13}N)\equiv  0$ (mod 7).

(2)  Let $k=6\alpha +5$ in (\ref{iterate}). In this case, we have ${{p}^{6\alpha +6}}\equiv 1$ (mod $7$). Hence from (\ref{iterate}) we obtain
\[\overline{p}_{3}(7p^{12\alpha +11}N)\equiv \Big(\frac{p^{6\alpha+6}-p}{p-1}+1\Big)\overline{p}_{3}(7pN) \equiv 0 \pmod{7}.\]

(3) Let $n=N$ in (\ref{r3relation}). Since $p+1 \equiv 0$ (mod 7), the congruence follows immediately.
\end{proof}

\begin{thm}
For any integers $n\ge 0$ and $\alpha \ge 1$, we have
\[\overline{p}_{3}(7^4\cdot n) \equiv \overline{p}_{3}(7^2\cdot n) \pmod{7}\]
and
\[{{\overline{p}}_{3}}\big({{7}^{2\alpha +1}}(7n+3)\big)\equiv {{\overline{p}}_{3}}\big({{7}^{2\alpha +1}}(7n+5)\big)\equiv {{\overline{p}}_{3}}\big({{7}^{2\alpha +1}}(7n+6)\big)\equiv 0 \pmod{7}.\]
\end{thm}
\begin{proof}
By Theorem \ref{p3r3} and (\ref{r3start}) we obtain
\begin{equation}\label{p3new}
\overline{p}_{3}(7\cdot p^2n)+\Big(\frac{-n}{p}\Big)\overline{p}_{3}(7n)+p\overline{p}_{3}(7n/p^2) \equiv (p+1)\overline{p}_{3}(7n) \pmod{7}.
\end{equation}
Let $p=7$ and we replace $n$ by $7n$. We deduce that $\overline{p}_{3}(7^4\cdot n) \equiv \overline{p}_{3}(7^2\cdot n)$ (mod 7).

Let $p=7$ and $n =7m+r,r\in \{3,5,6\}$ in (\ref{p3new}). Note that $\big(\frac{-r}{7}\big)=1$, it follows from (\ref{p3new}) that ${\overline{p}_{3}}(7^3(7m+r))\equiv 0$ (mod $7$). Combining this congruence with the first congruence relation, by induction on $\alpha$,  we complete our proof of the theorem.
\end{proof}

\begin{cor}\label{mod7cor}
${{\overline{p}}_{3}}(n)$ is divisible by $7$ for at least $1/784$ of all the time.
\end{cor}
\begin{proof}
Note that for different pairs $(\alpha, r)$, where $\alpha \ge 1$ and $r\in \{3,5,6\}$, the arithmetic sequences $\big\{{{7}^{2\alpha +1}}(7n+r):n=0,1,2,\cdots \big\}$ are disjoint, and they account for
\[3\Big(\frac{1}{{{7}^{4}}}+\frac{1}{{{7}^{6}}}+\cdots +\frac{1}{{{7}^{2\alpha +2}}}+\cdots \Big) =\frac{1}{784}\]
of all nonnegative integers.
\end{proof}

\begin{thm}\label{mod11}
For any integers $n\ge 0$, we have
\[{{\overline{p}}_{3}}(11n)\equiv {{(-1)}^{n}}{{r}_{7}}(n) \pmod{11}.\]
\end{thm}
\begin{proof}
By definition we have
\[\varphi {{(q)}^{11}}\sum\limits_{n\ge 0}{{{\overline{p}}_{3}}(n){{(-q)}^{n}}}=\varphi {{(q)}^{8}}=\sum\limits_{n\ge 0}{{{r}_{8}}(n){{q}^{n}}}.\]
By Lemma \ref{basic}, we have  $\varphi {{(q)}^{11}}\equiv \varphi ({{q}^{11}})$ (mod  $11$). Hence
\[\varphi ({{q}^{11}})\sum\limits_{n\ge 0}{{{\overline{p}}_{3}}(n){{(-q)}^{n}}}\equiv \sum\limits_{n\ge 0}{{{r}_{8}}(n){{q}^{n}}} \pmod{11}.\]
Collecting all the terms of the form ${{q}^{11n}}$ on both sides, then replacing ${{q}^{11}}$ by $q$ and applying Lemma \ref{r48modp} with $p=11$,  we get
\[\varphi (q)\sum\limits_{n\ge 0}{{{\overline{p}}_{3}}(11n){{(-q)}^{n}}}\equiv \sum\limits_{n\ge 0}{{{r}_{8}}(11n){{q}^{n}}} \equiv \sum\limits_{n\ge 0}{{{r}_{8}}(n){{q}^{n}}}={\varphi (q)}^{8} \pmod{11}.\]
Hence
\[\sum\limits_{n\ge 0}{{{\overline{p}}_{3}}(11n){{(-q)}^{n}}}\equiv \varphi {{(q)}^{7}}=\sum\limits_{n\ge 0}{{{r}_{7}}(n){{q}^{n}}} \pmod{11}.\]
The theorem follows by comparing the coefficients of ${{q}^{n}}$ on  both sides.
\end{proof}

\begin{thm}\label{twomod11}
Let $p\ge 3$ be a prime, and $N$ be  a positive integer which is coprime to $p$. Let  $\alpha $ be any nonnegative  integer. \\
(1) We have $\overline{p}_{3}(11^3 \cdot n) \equiv \overline{p}_{3}(11n)$ \text{\rm{(mod 11)}}. \\
(2) If $p\equiv 1,3,4,5,9$ \text{\rm{(mod $11$)}}, then ${{\overline{p}}_{3}}(11{{p}^{22\alpha +21}}N)\equiv 0$ \text{\rm{(mod  $11$)}}. \\
(3) If $p\equiv 2,6,7,8,10$ \text{\rm{(mod $11$)}}, then ${{\overline{p}}_{3}}(11{{p}^{4\alpha +3}}N)\equiv 0$ \text{\rm{(mod  $11$)}}.
\end{thm}
\begin{proof}
From  \cite[Lemma 5.1]{Cooper} we have
\begin{equation}\label{r7start}
r_7(p^2n)=\Big(p^5-p^2\Big(\frac{-n}{p}\Big)+1\Big)r_{7}(n)-p^5r_{7}(n/p^2),
\end{equation}
here as before, $(\frac{\cdot}{p})$ is the Legendre symbol and we take $r_{7}(n/p^2)=0$ unless $p^2|n$.
Replacing $n$ by $np$, we see that
\[r_7(p^3n)=(p^5+1)r_{7}(pn)-p^5r_{7}(n/p).\]
Hence by Theorem \ref{mod11} we have
\begin{equation}\label{r7re}
\overline{p}_{3}(11p^3n)\equiv (p^5+1)\overline{p}_{3}(11pn)-p^5\overline{p}_{3}(11n/p) \pmod{11}.
\end{equation}
From (\ref{r7re}), by iterating we see that
\[\overline{p}_{3}(11p^{2k+1}n) \equiv \frac{p^5(p^{5k}-1)}{p^5-1}\Big(\overline{p}_{3}(11pn)-\overline{p}_{3}(11n/p)\Big)+\overline{p}_{3}(11pn) \pmod{11}.\]
Let $n=N$, which is coprime to $p$. We get $\overline{p}_{3}(11n/p)=0$ and thus
\begin{equation}\label{iterate11}
\overline{p}_{3}(11p^{2k+1}N) \equiv \Big(\frac{p^5(p^{5k}-1)}{p^5-1}+1\Big)\overline{p}_{3}(11pN) \pmod{11}.
\end{equation}

(1) Let $p=11$ in (\ref{r7start}). We deduce that $r_{7}(11^2\cdot n)\equiv r_{7}(n)$ (mod 11). By Theorem \ref{mod11} we get the desired congruence we want.

(2) Let $k=11\alpha +10$ in (\ref{iterate11}). Note that $p\equiv 1,3,4,5,9$ (mod $11$) implies ${{p}^{5}}\equiv 1$ (mod  $11$). We have
\[\frac{{{p}^{55\alpha +50}}-1}{{{p}^{5}}-1}=1+{{p}^{5}}+{{p}^{10}}+\cdots +{{p}^{5(11\alpha +9)}}\equiv 10 \pmod{11}.\]
It follows from (\ref{iterate}) that $\overline{p}_{3}(11p^{22\alpha+21}N)\equiv 0$ (mod 11).\\

(3)  Let $k=2\alpha +1$ in (\ref{iterate11}).  Note that $p\equiv 2,6,7,8,10$ (mod $11$) implies ${{p}^{5}}\equiv -1$  (mod  $11$). It follows from (\ref{iterate11}) that $\overline{p}_{3}(11p^{4\alpha+3}N) \equiv 0$ (mod 11).
\end{proof}

For example,  let  $p=7$ and $N=7n+r$, $r \in \{1,2,3,4,5,6\}$ in part (2) of Theorem \ref{twomod11}, we have
 \[{{\overline{p}}_{3}}\big(11\cdot {{7}^{4\alpha +3}}(7n+r)\big)\equiv 0 \pmod{11}.\]
By a similar argument to the proof of Corollary \ref{mod7cor}, we can prove
\begin{cor}
${{\overline{p}}_{3}}(n)$ is divisible by $11$ for at least $1/4400$ of all the time.
\end{cor}

% ----------------------------------------------------------------------------------------------------------------------------------------
%---------------------------------------------------------------------------------------------------------------------------------------
\bibliographystyle{amsplain}

\end{document}